\documentclass{bmcart}

\usepackage[utf8]{inputenc} 
\usepackage{amsmath,amsthm,latexsym,amssymb,mathrsfs}

\usepackage{graphicx}

\usepackage{multicol,amsmath,amssymb,amstext,graphicx}
\usepackage{setspace}
\usepackage{esint}
\usepackage{subfig}

\usepackage{lipsum}

\startlocaldefs
\usepackage{enumitem}		
\usepackage{braille}

\@ifundefined{lettrine}{\usepackage{lettrine}}{}
 \let\footnote=\endnote
\theoremstyle{plain}
\ifx\thechapter\undefined
\newtheorem{thm}{\protect\theoremname}
\else
\newtheorem{thm}{\protect\theoremname}[chapter]
\fi
  \theoremstyle{plain}
  \newtheorem{lem}[thm]{\protect\lemmaname}
  \theoremstyle{remark}
  \newtheorem{rem}[thm]{\protect\remarkname}

\@ifundefined{showcaptionsetup}{}{%
 \PassOptionsToPackage{caption=false}{subfig}}
\usepackage{subfig}

\makeatother

  \providecommand{\lemmaname}{Lemma}
  \providecommand{\remarkname}{Remark}
\providecommand{\theoremname}{Theorem}

\endlocaldefs

\begin{document}

\begin{frontmatter}

\begin{fmbox}
\dochead{Research}


\title{A numerical approach to approximation for a nonlinear ultraparabolic
equation}


\author[
   addressref={aff1},                   
   corref={aff1},                       
   email={vakhoa.hcmus@gmail.com}   
]{\inits{VAK}\fnm{Vo Anh} \snm{Khoa}}
\author[
   addressref={aff1},
   email={tronglanmath@gmail.com}
]{\inits{LTL}\fnm{Le Trong} \snm{Lan}}
\author[
   addressref={aff1,aff2},
   email={yenngoc0202@gmail.com}
]{\inits{NTYN}\fnm{Nguyen Thi Yen} \snm{Ngoc}}
\author[
   addressref={aff1,aff3},
   email={thnguyen2683@gmail.com}
]{\inits{NTH}\fnm{Nguyen Huy} \snm{Tuan}}


\address[id=aff1]{
  \orgname{Department of Mathematics and Computer Science, Ho Chi Minh City University of Science}, 
  \street{227 Nguyen Van Cu Street, District 5},                     %
  \city{Ho Chi Minh City},                              
  \cny{Vietnam}                                    
}
\address[id=aff2]{%
  \orgname{Institut de Mathématiques de Bordeaux, Université de Bordeaux},
  \street{351 Cours de la Libération},
  \city{33405 Talence Cedex},
  \cny{France}
}
\address[id=aff3]{%
  \orgname{Saigon Institute for Computational Science and Technology},
  \street{Quang Trung Software City, District 12},
  \city{Ho Chi Minh City},
  \cny{Vietnam}
}



\end{fmbox}


\begin{abstractbox}

\begin{abstract} 
In this paper, our aim is to study a numerical method for an ultraparabolic equation with nonlinear source function. Mathematically,
the bibliography on initial-boundary value problems for ultraparabolic
equations is not extensive although the problems have many applications
related to option pricing, multi-parameter Brownian motion, population
dynamics and so forth. In this work, we present the approximate solution
by virtue of finite difference scheme and Fourier series. For the
linear case, we give the approximate solution and obtain a stability
result. For the nonlinear case, we use an iterative scheme by linear
approximation to get the approximate solution and obtain error estimates.
Some numerical examples are given to demonstrate the efficiency of
the method.


\end{abstract}




\end{abstractbox}
%

\end{frontmatter}



\section{Introduction}

Let $\mathcal{H}$ be a Hilbert space with the inner product $\left\langle .,.\right\rangle $
and the norm $\left\Vert .\right\Vert $. In this paper, we study
the following ultraparabolic equation

\begin{equation}
\frac{\partial}{\partial t}u\left(t,s\right)+\frac{\partial}{\partial s}u\left(t,s\right)+\mathcal{L}u\left(t,s\right)=f\left(u\left(t,s\right),t,s\right),\quad\left(t,s\right)\in\left(0,T\right)\times\left(0,T\right),\label{eq:1}
\end{equation}

associated with the initial conditions

\begin{equation}
u\left(0,s\right)=\alpha\left(s\right),\quad u\left(t,0\right)=\beta\left(t\right).\label{eq:2}
\end{equation}

where $\mathcal{L}:\mathcal{D}\left(\mathcal{L}\right)\subset\mathcal{H}\to\mathcal{H}$
is a positive-definite, self-adjoint operator with compact inverse
on $\mathcal{H}$ and $\alpha,\beta$ are known smooth functions satisfying
$\alpha\left(0\right)=\beta\left(0\right)$ for compatibility at $\left(t,s\right)=\left(0,0\right)$
and $f$ is a nonlinear source function satisfying some conditions
which will be fully presented in the next section.

The problem (\ref{eq:1})-(\ref{eq:2}) involving multi-dimensional
time variables is called the initial-boundary value problem for ultraparabolic
equation. The ultraparabolic equation has many applications in mathematical
finance (e.g. \cite{key-10}), physics (such as multi-parameter Brownian
motion \cite{key-22}) and biological model. Among many applications,
the equation (\ref{eq:1}) arises as a mathematical model of population
dynamics, for instance, the dynamics of the age structure of an isolated
at the distinct moments of astronomical or biological time and $u\left(t,s\right)$
in this application plays a role as the number of individuals of age
$s$ in the population at time $t$. The study of ultraparabolic equation
for population dynamics can be found in some papers such as \cite{key-9,key-11}.
In particular, Kozhanov \cite{key-11} studied the existence and uniqueness
of regular solutions and its properties for an ultraparabolic model
equation in the form of

\[
\frac{\partial u}{\partial t}+\frac{\partial u}{\partial s}-\Delta u+h\left(x,t,s\right)u+u\mathcal{A}u=f\left(x,t,s\right),
\]

where $\Delta$ is Laplace operator, $\mathcal{A}$ is a nonlocal
linear operator. In the same work, the authors Deng and Hallam in
\cite{key-9} considered the age structured population problem formed

\[
\frac{\partial u}{\partial t}+\frac{\partial u}{\partial s}-\nabla\cdot\left(k\nabla u-qu\right)=-\mu u,
\]

associated with non-locally integro-type initial-bounded conditions.

The ultraparabolic equation is also studied in many other aspects.
In the phase of inverse problems, Lorenzi \cite{key-17} studied the
well-posedness of a class of an inverse problem for ultraparabolic
partial integrodifferential equations of Volterra type. Very recently,
Zouyed and Rebbani \cite{key-6} proposed the modified quasi-boundary
value method to regularize the equation (\ref{eq:1}) in homogeneous
backward case in a class of ill-posed problems. For another studies
regarding the properties of solutions of abstract ultraparabolic equations,
we can find many papers and some of them are refered to \cite{key-12,key-13,key-15,key-16,key-18,key-20}.

Even though the numerical method for such a problem is studied long
time ago, it is still very limited. We only find some papers, such
as \cite{key-4,key-7,key-19}. The authors Akrivis, Crouzeix and Thomée
\cite{key-4} investigated a backward Euler scheme and second-order
box-type finite difference procedure to numerically approximate the
solution to the Dirichlet problem for the ultraparabolic equation
(\ref{eq:1})-(\ref{eq:2}) in two different time intervals with the
Laplace operator $\mathcal{L}=-\Delta$ and source function $f\equiv0$.
Recently, Ashyralyev and Yilmaz \cite{key-7} constructed the first
and second order difference schemes to approximate the problem (\ref{eq:1})-(\ref{eq:2})
for strongly positive operator and obtained some fundamental stability
results. On the other hand, Marcozzi et al. \cite{key-19} developed
an adaptive method-of-lines extrapolation discontinuous Galerkin method
for an ultraparabolic model equation given by

\[
\frac{\partial u}{\partial t}+a\left(x\right)\frac{\partial u}{\partial s}-\Delta u=f\left(x,t,s\right),
\]

with a certain application to the price of an Asian call option.

However, we can see that most of papers for numerical methods aim
to study linear cases. Equivalently, numerical methods for nonlinear
equations are investigated rarely. Therefore, in this paper we shall
study the model problem (\ref{eq:1})-(\ref{eq:2}) in the numerical
angle for the smooth solution. From the idea of finite difference
scheme and conveying a fundamental result in operator theory, we construct
an approximate solution for the nonhomogeneous equation in terms of
Fourier series. Combining the same technique and linear approximation,
the approximate solution for the nonlinear case is established.

The rest of the paper is organized as follows. In Section 2, we shall
consider the linear nonhomogeneous problem (\ref{eq:1})-(\ref{eq:2})
under a result of presentation of discretization solution in multi-dimensional
problem. The nonlinear problem is considered in Section 3 and an iterative
scheme is showed. Finally, four numerical examples are implemented
in Section 4 to verify the effect of the method.

\section{The linear nonhomogeneous problem}

In this section, we shall introduce the suitable discrete operator
used in the time discretization. In order to define the discrete operator
involved in the equation in the problem (\ref{eq:1})-(\ref{eq:2}),
we consider the multi-dimensional problem given by

\begin{equation}
\left(\sum_{i=1}^{d}\frac{\partial}{\partial t_{i}}\right)u\left(t_{1},t_{2},...,t_{d}\right)+\mathcal{L}u\left(t_{1},t_{2},...,t_{d}\right)=f\left(t_{1},t_{2},...,t_{d}\right),\quad t_{i}\in\left(0,T\right),1\le i\le d,\label{eq:3}
\end{equation}

associated with $d$ initial conditions

\begin{equation}
u\left(0,t_{2},...,t_{d}\right)=\alpha_{1}\left(t_{2},t_{3},...,t_{d}\right),...,u\left(t_{1},t_{2},...,0\right)=\alpha_{d}\left(t_{1},t_{2},...,t_{d-1}\right),\label{eq:4}
\end{equation}

for $t_{i}\in\left[0,T\right],1\le i\le d$ are $d$ time variables
and $d\ge2$.

Since $\mathcal{L}^{-1}$ is compact, the operator $\mathcal{L}$
admits an orthonormal eigenbasis $\left\{ \phi_{n}\right\} _{n\ge1}$
for $\mathcal{H}$ and eigenvalues $\dfrac{1}{\lambda_{n}}$ of $\mathcal{L}^{-1}$
such that $\mathcal{L}^{-1}\phi_{n}=\dfrac{1}{\lambda_{n}}\phi_{n}$
and $0<\lambda_{1}\le\lambda_{2}\le...{\displaystyle \lim_{n\to\infty}\lambda_{n}}=\infty$.
Thus, we have from (\ref{eq:3}) that

\begin{equation}
\left(\sum_{i=1}^{d}\frac{\partial}{\partial t_{i}}\right)\left\langle u\left(t_{1},t_{2},...,t_{d}\right),\phi_{n}\right\rangle +\lambda_{n}\left\langle u\left(t_{1},t_{2},...,t_{d}\right),\phi_{n}\right\rangle =\left\langle f\left(t_{1},t_{2},...,t_{d}\right),\phi_{n}\right\rangle ,\label{eq:5}
\end{equation}

and the conditions (\ref{eq:4}) can be transformed into

\begin{equation}
\left\langle u\left(0,t_{2},...,t_{d}\right),\phi_{n}\right\rangle =\left\langle \alpha_{1}\left(t_{2},t_{3},...,t_{d}\right),\phi_{n}\right\rangle ,...,\left\langle u\left(t_{1},t_{2},...,0\right),\phi_{n}\right\rangle =\left\langle \alpha_{d}\left(t_{1},t_{2},...,t_{d-1}\right),\phi_{n}\right\rangle .
\end{equation}

By putting

\[
u_{n}\left(t_{1},t_{2},...,t_{d}\right)=\left\langle u\left(t_{1},t_{2},...,t_{d}\right),\phi_{n}\right\rangle ,\quad f_{n}\left(t_{1},t_{2},...,t_{d}\right)=\left\langle f\left(t_{1},t_{2},...,t_{d}\right),\phi_{n}\right\rangle ,
\]

\[
\alpha_{1,n}\left(t_{2},t_{3},...,t_{d}\right)=\left\langle \alpha_{1}\left(t_{2},t_{3},...,t_{d}\right),\phi_{n}\right\rangle ,...,\alpha_{d,n}\left(t_{1},t_{2},...,t_{d-1}\right)=\left\langle \alpha_{d}\left(t_{1},t_{2},...,t_{d-1}\right),\phi_{n}\right\rangle ,
\]

we shall establish an discrete problem by knowledge of finite difference
scheme.

By multiplying both sides of (\ref{eq:5}) by $\mu_{n}\left(t_{1},t_{2},...,t_{d}\right):=\mbox{exp}\left({\displaystyle \frac{\lambda_{n}}{d}\sum_{i=1}^{d}}t_{i}\right)$,
we get

\[
\left(\sum_{i=1}^{d}\frac{\partial}{\partial t_{i}}\right)\left(u_{n}\left(t_{1},t_{2},...,t_{d}\right)\mu_{n}\left(t_{1},t_{2},...,t_{d}\right)\right)=f_{n}\left(t_{1},t_{2},...,t_{d}\right)\mu_{n}\left(t_{1},t_{2},...,t_{d}\right).
\]

By putting

\begin{eqnarray*}
v_{n}\left(t_{1},t_{2},...,t_{d}\right) & = & u_{n}\left(t_{1},t_{2},...,t_{d}\right)\mu_{n}\left(t_{1},t_{2},...,t_{d}\right)\\
F_{n}\left(t_{1},t_{2},...,t_{d}\right) & = & f_{n}\left(t_{1},t_{2},...,t_{d}\right)\mu_{n}\left(t_{1},t_{2},...,t_{d}\right),
\end{eqnarray*}

we thus have the problem (\ref{eq:3})-(\ref{eq:4}) in a new form.

\begin{equation}
\left(\sum_{i=1}^{d}\frac{\partial}{\partial t_{i}}\right)v_{n}\left(t_{1},t_{2},...,t_{d}\right)=F_{n}\left(t_{1},t_{2},...,t_{d}\right),\label{eq:7}
\end{equation}

associated with the initial conditions

\begin{eqnarray}
v_{n}\left(0,t_{2},...,t_{d}\right) & = & \beta_{1,n}\left(t_{2},t_{3},...,t_{d}\right):=\alpha_{1,n}\left(t_{2},t_{3},...,t_{d}\right)\mu_{n}\left(0,t_{2},...,t_{d}\right),\nonumber \\
\vdots & \vdots & \vdots\nonumber \\
v_{n}\left(t_{1},t_{2},...,0\right) & = & \beta_{d,n}\left(t_{1},t_{2},...,t_{d-1}\right):=\alpha_{d,n}\left(t_{1},t_{2},...,t_{d-1}\right)\mu_{n}\left(t_{1},t_{2},...,0\right).\label{eq:8}
\end{eqnarray}

Then, we treat the transport-shape problem (\ref{eq:7})-(\ref{eq:8})
by finite difference scheme. However, note that standard finite difference
scheme may suffer an unstable behaviour in approximate aspect. To
avoid the behaviour and lead to our discrete operator, we shall take
the equivalent mesh-width in time $\Delta t_{i}=\omega$ for all $1\le i\le d$.
Also, the source term $f$ in this case is discretized at center nodes
to obtain a better order accurate in the first order scheme. They
offer stable discretizations of the problem well and then preserve
the discrete approach.

For $v_{n}^{k_{1},k_{2},...,k_{d}}:=v_{n}\left(\left(t_{1}\right)_{k_{1}},\left(t_{2}\right)_{k_{2}},...,\left(t_{d}\right)_{k_{d}}\right)$,
we get

\begin{eqnarray}
\frac{v_{n}^{k_{1},k_{2},...,k_{d}}-v_{n}^{k_{1}-1,k_{2},...,k_{d}}}{\omega}+\frac{v_{n}^{k_{1}-1,k_{2},...,k_{d}}-v_{n}^{k_{1}-1,k_{2}-1,...,k_{d}}}{\omega}+...\nonumber \\
...+\frac{v_{n}^{k_{1}-1,k_{2}-1,...,k_{d}}-v_{n}^{k_{1}-1,k_{2}-1,...,k_{d}-1}}{\omega}=F_{n}^{k_{1},k_{2},...,k_{d}}\label{eq:9}
\end{eqnarray}

\[
F_{n}^{k_{1},k_{2},...,k_{d}}=f_{n}\left(\left(t_{1}\right)_{k_{1}}-\frac{\omega}{2},\left(t_{2}\right)_{k_{2}}-\frac{\omega}{2},...,\left(t_{d}\right)_{k_{d}}-\frac{\omega}{2}\right)\mu_{n}\left(\left(t_{1}\right)_{k_{1}},\left(t_{2}\right)_{k_{2}},...,\left(t_{d}\right)_{k_{d}}\right),
\]

for $1\le k_{i}\le M,1\le i\le d$,

\begin{eqnarray}
v_{n}^{0,k_{2},...,k_{d}} & = & \beta_{1,n}^{k_{2},k_{3},...,k_{d}}:=\alpha_{1,n}\left(\left(t_{2}\right)_{k_{2}},\left(t_{3}\right)_{k_{3}},...,\left(t_{d}\right)_{k_{d}}\right)\mu_{n}\left(0,\left(t_{2}\right)_{k_{2}},...,\left(t_{d}\right)_{k_{d}}\right)\nonumber \\
\vdots & \vdots & \vdots\nonumber \\
v_{n}^{k_{1},k_{2},...,0} & = & \beta_{d,n}^{k_{1},k_{2},...,k_{d-1}}:=\alpha_{d,n}\left(\left(t_{1}\right)_{k_{1}},\left(t_{2}\right)_{k_{2}},...,\left(t_{d-1}\right)_{k_{d-1}}\right)\mu_{n}\left(\left(t_{1}\right)_{k_{1}},\left(t_{2}\right)_{k_{2}},...,0\right),\nonumber \\
\label{eq:10}
\end{eqnarray}

for $0\le k_{i}\le M$ and where $\left(t_{i}\right)_{k_{i}}=k_{i}\omega,1\le i\le d$
and $M\omega=T$.

By simple calculation, it follows from (\ref{eq:9}) that the local
discrete operator is

\begin{equation}
v_{n}^{k_{1},k_{2},...,k_{d}}=\omega F_{n}^{k_{1},k_{2},...,k_{d}}+v_{n}^{k_{1}-1,k_{2}-1,...,k_{d}-1}.\label{eq:11}
\end{equation}

\begin{lem}
For every $n\in\mathbb{N}$, if $v_{n}^{k_{1},k_{2},...,k_{d}}$ satisfies
(\ref{eq:11}) for all $k_{i}\in\left[1,M\right],1\le i\le d$, then
one has

\begin{equation}
v_{n}^{k_{1},k_{2},...,k_{d}}=\omega\sum_{l=1}^{p}F_{n}^{k_{1}-p+l,k_{2}-p+l,...,k_{d}-p+l}+v_{n}^{k_{1}-p,k_{2}-p,...,k_{d}-p},\label{eq:12}
\end{equation}

for all $p\in\mathbb{N}$.\end{lem}
\begin{proof}
We shall prove (\ref{eq:12}) by induction. We can see that it holds
for $p=1$.

For $p=2$, we have

\begin{eqnarray*}
v_{n}^{k_{1},k_{2},...,k_{d}} & = & \omega F_{n}^{k_{1},k_{2},...,k_{d}}+\omega f_{n}^{k_{1}-1,k_{2}-1,...,k_{d}-1}+v_{n}^{k_{1}-2,k_{2}-2,...,k_{d}-2}\\
 & = & \omega F_{n}^{k_{1},k_{2},...,k_{d}}+v_{n}^{k_{1}-1,k_{2}-1,...,k_{d}-1}.
\end{eqnarray*}

Thus, (\ref{eq:12}) holds for $p=2$.

Now, we assume that (\ref{eq:12}) holds for $p=r$. It means that

\[
v_{n}^{k_{1},k_{2},...,k_{d}}=\omega\sum_{l=1}^{r}F_{n}^{k_{1}-r+l,k_{2}-r+l,...,k_{d}-r+l}+v_{n}^{k_{1}-r,k_{2}-r,...,k_{d}-r}.
\]

We shall prove that it aslo holds for $p=r+1$. Indeed, we get

\begin{eqnarray*}
v_{n}^{k_{1},k_{2},...,k_{d}} & = & \omega\sum_{l=1}^{r}F_{n}^{k_{1}-r+l,k_{2}-r+l,...,k_{d}-r+l}\\
 &  & +\omega F_{n}^{k_{1}-r,k_{2}-r,...,k_{d}-r}+v_{n}^{k_{1}-r-1,k_{2}-r-1,...,k_{d}-r-1}\\
 & = & \omega\sum_{l=1}^{r+1}F_{n}^{k_{1}-r-1+l,k_{2}-r-1+l,...,k_{d}-r-1+l}+v_{n}^{k_{1}-r-1,k_{2}-r-1,...,k_{d}-r-1}.
\end{eqnarray*}

Therefore, (\ref{eq:12}) holds for $p=r+1$. By induction, we completely
finish the proof.
\end{proof}
From (\ref{eq:12}), we shall obtain the discrete solution by Fourier
series. It should be stated that the discrete solution is, in multi-dimensional,
involved by many situations according to the set $\mathcal{E}_{d}=\left\{ k_{1},k_{2},...,k_{d}\right\} $,
more exactly $d!$ situations. As introduced, in this paper we aim
to consider the ultraparabolic problem with two time dimension since
there are many studies on this problem in real application. Hence,
the solution of the discrete problem of (\ref{eq:1})-(\ref{eq:2})
in linear nonhomogeneous case is

\[
v_{n}^{k_{1},k_{2}}=\omega\sum_{l=1}^{p}F_{n}^{k_{1}-p+l,k_{2}-p+l}+v_{n}^{k_{1}-p,k_{2}-p},
\]

and its explicit form is given as follows.

If $k_{1}>k_{2}$, we replace $p$ by $k_{2}$ to get

\begin{eqnarray}
v_{n}^{k_{1},k_{2}} & = & \omega\sum_{l=1}^{k_{2}}F_{n}^{k_{1}-k_{2}+l,l}+v_{n}^{k_{1}-k_{2},0}\nonumber \\
 & = & \omega\sum_{l=1}^{k_{2}}F_{n}^{k_{1}-k_{2}+l,l}+\beta_{n}^{k_{1}-k_{2}}\mu_{n}^{k_{1}-k_{2},0}.\label{eq:13}
\end{eqnarray}

Similarly, for $k_{2}>k_{1}$ we replace $p$ by $k_{1}$ to obtain

\begin{equation}
v_{n}^{k_{1},k_{2}}=\omega\sum_{l=1}^{k_{1}}F_{n}^{l,k_{2}-k_{1}+l}+\alpha_{n}^{k_{2}-k_{1}}\mu_{n}^{0,k_{2}-k_{1}}.\label{eq:14}
\end{equation}

From (\ref{eq:13})-(\ref{eq:14}), we conclude the discrete solution
for the two-time-variable ultraparabolic (\ref{eq:1})-(\ref{eq:2})
in linear nonhomogeneous case is

\begin{eqnarray}
u^{k_{1},k_{2}} & = & \sum_{n=1}^{\infty}v_{n}^{k_{1},k_{2}}\left(\mu_{n}^{k_{1},k_{2}}\right)^{-1}\phi_{n}\nonumber \\
 & = & \begin{cases}
\sum_{n=1}^{\infty}\omega\sum_{l=1}^{k_{2}}\mu_{n}^{k_{1}-k_{2}+l,l}\left(\mu_{n}^{k_{1},k_{2}}\right)^{-1}f_{n}^{k_{1}-k_{2}+l,l}\phi_{n}\\
+\sum_{n=1}^{\infty}\left(\mu_{n}^{k_{1},k_{2}}\right)^{-1}\beta_{n}^{k_{1}-k_{2}}\mu_{n}^{k_{1}-k_{2},0}\phi_{n} & ,k_{1}>k_{2},\\
\sum_{n=1}^{\infty}\omega\sum_{l=1}^{k_{1}}\mu_{n}^{l,k_{2}-k_{1}+l}\left(\mu_{n}^{k_{1},k_{2}}\right)^{-1}f_{n}^{l,k_{2}-k_{1}+l}\phi_{n}\\
+\sum_{n=1}^{\infty}\left(\mu_{n}^{k_{1},k_{2}}\right)^{-1}\alpha_{n}^{k_{2}-k_{1}}\mu_{n}^{0,k_{2}-k_{1}}\phi_{n} & ,k_{2}>k_{1}.
\end{cases}\label{eq:15}
\end{eqnarray}

Furthermore, we obtain a stability result in the following theorem.
\begin{thm}
Let $u^{k_{1},k_{2}}$ in (\ref{eq:15}) be the discrete solution
of the problem (\ref{eq:1})-(\ref{eq:2}) in linear nonhomogeneous
case, then there exists a positive constant $C_{T}$ independent of
$k_{1},k_{2}$ and $\omega$ such that

\[
\sup_{1\le k_{1},k_{2}\le M}\left\Vert u^{k_{1},k_{2}}\right\Vert ^{2}\le C_{T}\left(\sup_{1\le k_{1},k_{2}\le M}\left\Vert f^{k_{1},k_{2}}\right\Vert ^{2}+\sup_{0\le k_{1}\le M}\left\Vert \alpha^{k_{2}}\right\Vert ^{2}+\sup_{0\le k_{1}\le M}\left\Vert \beta^{k_{1}}\right\Vert ^{2}\right).
\]
\end{thm}
\begin{proof}
Since $\mu_{n}^{k_{1}-k_{2}+l,l}\le\mu_{n}^{k_{1},k_{2}}$ and $\mu_{n}^{k_{1}-k_{2},0}\le\mu_{n}^{k_{1},k_{2}}$,
by using Parseval's identity we have

\begin{eqnarray}
\left\Vert u^{k_{1},k_{2}}\right\Vert ^{2} & = & \sum_{n=1}^{\infty}\left|v_{n}^{k_{1},k_{2}}\left(\mu_{n}^{k_{1},k_{2}}\right)^{-1}\right|^{2}\nonumber \\
 & \le & 2\omega^{2}\sum_{n=1}^{\infty}\left|\sum_{l=1}^{k_{2}}\mu_{n}^{k_{1}-k_{2}+l,l}\left(\mu_{n}^{k_{1},k_{2}}\right)^{-1}f_{n}^{k_{1}-k_{2}+l,l}\right|^{2}\nonumber \\
 &  & +2\sum_{n=1}^{\infty}\left|\left(\mu_{n}^{k_{1},k_{2}}\right)^{-1}\mu_{n}^{k_{1}-k_{2},0}\beta_{n}^{k_{1}-k_{2}}\right|^{2}\nonumber \\
 & \le & 2k_{2}^{2}\omega^{2}\sup_{1\le k_{1},k_{2}\le M}\sum_{n=1}^{\infty}\left|f_{n}^{k_{1},k_{2}}\right|^{2}+2\sum_{n=1}^{\infty}\left|\beta_{n}^{k_{1}-k_{2}}\right|^{2}\nonumber \\
 & \le & 2T^{2}\sup_{1\le k_{1},k_{2}\le M}\left\Vert f^{k_{1},k_{2}}\right\Vert ^{2}+2\left\Vert \beta^{k_{1}-k_{2}}\right\Vert ^{2}\nonumber \\
 & \le & C_{T}\left(\sup_{1\le k_{1},k_{2}\le M}\left\Vert f^{k_{1},k_{2}}\right\Vert ^{2}+\sup_{0\le k_{1}\le M}\left\Vert \beta^{k_{1}}\right\Vert ^{2}\right),\label{eq:16}
\end{eqnarray}

for $k_{1}>k_{2}$.

Similarly, for $k_{2}>k_{1}$ we also have

\begin{eqnarray}
\left\Vert u^{k_{1},k_{2}}\right\Vert ^{2} & \le & 2T^{2}\sup_{1\le k_{1},k_{2}\le M}\left\Vert f^{k_{1},k_{2}}\right\Vert ^{2}+2\left\Vert \alpha^{k_{2}-k_{1}}\right\Vert ^{2}\nonumber \\
 & \le & C_{T}\left(\sup_{1\le k_{1},k_{2}\le M}\left\Vert f^{k_{1},k_{2}}\right\Vert ^{2}+\sup_{0\le k_{2}\le M}\left\Vert \alpha^{k_{2}}\right\Vert ^{2}\right).\label{eq:17}
\end{eqnarray}

Combining (\ref{eq:16}) and (\ref{eq:17}), we conclude that

\[
\sup_{1\le k_{1},k_{2}\le M}\left\Vert u^{k_{1},k_{2}}\right\Vert ^{2}\le C_{T}\left(\sup_{1\le k_{1},k_{2}\le M}\left\Vert f^{k_{1},k_{2}}\right\Vert ^{2}+\sup_{0\le k_{2}\le M}\left\Vert \alpha^{k_{2}}\right\Vert ^{2}+\sup_{0\le k_{1}\le M}\left\Vert \beta^{k_{1}}\right\Vert ^{2}\right),
\]

which gives the desired result.\end{proof}
\begin{rem}
The stability result for the ultraparabolic problem in multi-time
dimension (\ref{eq:9})-(\ref{eq:10}) can be obtained in the same
way according to the situations the discrete solution has. Particularly,
one has
\end{rem}
\[
\sup_{1\le k_{i}\le M,1\le i\le d}\left\Vert u^{k_{1},k_{2},...,k_{d}}\right\Vert ^{2}\le C_{T}\left(\sup_{1\le k_{i}\le M,1\le i\le d}\left\Vert f^{k_{1},k_{2},...,k_{d}}\right\Vert ^{2}+\sum_{j=1}^{d}\sup_{0\le k_{i}\le M,1\le i\le d}\left\Vert \alpha_{j}^{k_{i}}\right\Vert ^{2}\right).
\]

\section{The nonlinear problem}

Until now, numerical methods for nonlinear ultraparabolic equations
are still very rare. From that point, we begin the section by considering
the ultraparabolic problem (\ref{eq:1})-(\ref{eq:2}) with the nonlinear
function $f$ satisfying Lipschitz condition. By simple calculation
analogous to the steps in linear nonhomogeneous case, we get the discrete
solution, then use linear approximation to get the explicit form of
the approximate solution. Particularly, the following problem is considered.

\[
\left(P\right):\quad\begin{cases}
\frac{\partial}{\partial t}u\left(t,s\right)+\frac{\partial}{\partial s}u\left(t,s\right)+\mathcal{L}u\left(t,s\right)=f\left(u\left(t,s\right),t,s\right),\\
u\left(0,s\right)=\alpha\left(s\right),\quad u\left(t,0\right)=\beta\left(t\right),
\end{cases}
\]

for the nonlinear source function $f$ satisfying the Lipschitz condition:

\begin{equation}
\left\Vert f\left(u,t,s\right)-f\left(v,t,s\right)\right\Vert \le K\left\Vert u-v\right\Vert ,\label{eq:18}
\end{equation}

where $K$ is a positive number independent of $u,v,t,s$.

On account of the orthonormal basis $\left\{ \phi_{n}\right\} _{n\ge1}$
admitted by $\mathcal{L}$ and corresponding eigenvalues $\lambda_{n}$,
the problem $\left(P\right)$ can be made in the following manner.

\begin{equation}
\begin{cases}
\frac{\partial}{\partial t}\left\langle u\left(t,s\right),\phi_{n}\right\rangle +\frac{\partial}{\partial s}\left\langle u\left(t,s\right),\phi_{n}\right\rangle +\lambda_{n}\left\langle u\left(t,s\right),\phi_{n}\right\rangle =\left\langle f\left(u\left(t,s\right),t,s\right),\phi_{n}\right\rangle ,\\
\left\langle u\left(0,s\right),\phi_{n}\right\rangle =\left\langle \alpha\left(s\right),\phi_{n}\right\rangle ,\quad\left\langle u\left(t,0\right),\phi_{n}\right\rangle =\left\langle \beta\left(t\right),\phi_{n}\right\rangle .
\end{cases}\label{eq:19}
\end{equation}

With $\mu_{n}\left(t,s\right)=\mbox{exp}\left(\dfrac{\lambda_{n}}{2}\left(t+s\right)\right)$,
the problem (\ref{eq:19}) is equivalent to the following problem.

\begin{equation}
\begin{cases}
\left(\frac{\partial}{\partial t}+\frac{\partial}{\partial s}\right)\left(\left\langle u\left(t,s\right),\phi_{n}\right\rangle \mu_{n}\left(t,s\right)\right)=\left\langle f\left(u\left(t,s\right),t,s\right),\phi_{n}\right\rangle \mu_{n}\left(t,s\right),\\
\left\langle u\left(0,s\right),\phi_{n}\right\rangle \mu_{n}\left(t,s\right)=\left\langle \alpha\left(s\right),\phi_{n}\right\rangle \mu_{n}\left(t,s\right),\\
\left\langle u\left(t,0\right),\phi_{n}\right\rangle \mu_{n}\left(t,s\right)=\left\langle \beta\left(t\right),\phi_{n}\right\rangle \mu_{n}\left(t,s\right).
\end{cases}\label{eq:20}
\end{equation}

For the numerical solution of this problem by finite difference scheme
as introduced in the above section, a uniform grid of mesh-points
$\left(t,s\right)=\left(t_{k},s_{m}\right)$ is showed. Here $t_{k}=k\omega$
and $s_{m}=m\omega$, where $k$ and $m$ are integers and $\omega$
the equivalent mesh-width in time $t$ and $s$. We shall seek an
discrete solution $u^{k,m}=u\left(t_{k},s_{m}\right)$ determined
by an equation obtained by replacing the time derivatives in (\ref{eq:20})
by difference quotients. The equation in (\ref{eq:20}) becomes

\begin{eqnarray}
\frac{\left\langle u^{k,m},\phi_{n}\right\rangle \mu_{n}^{k,m}-\left\langle u^{k-1,m},\phi_{n}\right\rangle \mu_{n}^{k-1,m}}{\omega}+\frac{\left\langle u^{k-1,m},\phi_{n}\right\rangle \mu_{n}^{k-1,m}-\left\langle u^{k-1,m-1},\phi_{n}\right\rangle \mu_{n}^{k-1,m-1}}{\omega}\nonumber \\
=\left\langle f\left(u^{k,m},t_{k},s_{m}\right),\phi_{n}\right\rangle \mu_{n}^{k,m},\nonumber \\
\label{eq:21}
\end{eqnarray}

and the initial conditions are

\[
\left\langle u\left(0,s_{m}\right),\phi_{n}\right\rangle \mu_{n}^{k,m}=\left\langle \alpha^{m},\phi_{n}\right\rangle \mu_{n}^{k,m},\quad\left\langle u\left(t_{k},0\right),\phi_{n}\right\rangle \mu_{n}^{k,m}=\left\langle \beta^{k},\phi_{n}\right\rangle \mu_{n}^{k,m},
\]

where $\mu_{n}^{k,m}=\mu_{n}\left(t_{k},s_{m}\right),\alpha^{m}=\alpha\left(s_{m}\right)$
and $\beta^{k}=\beta\left(t_{k}\right)$.

By induction, it follows from (\ref{eq:21}) that

\begin{eqnarray*}
\left\langle u^{k,m},\phi_{n}\right\rangle \mu_{n}^{k,m} & = & \omega\sum_{l=1}^{p}\left\langle f\left(u^{k-p+l,m-p+l},t_{k-p+l},s_{m-p+l}\right),\phi_{n}\right\rangle \mu_{n}^{k-p+l,m-p+l}\\
 &  & +\left\langle u^{k-p,m-p},\phi_{n}\right\rangle \mu_{n}^{k-p,m-p},
\end{eqnarray*}

for all $p\in\mathbb{N}$.

Thus, the explicit form of discrete solution of $\left(P\right)$
is obtained.

\begin{equation}
u^{k,m}=\begin{cases}
\sum_{n=1}^{\infty}\omega\sum_{l=1}^{m}\mu_{n}^{k-m+l,l}\left(\mu_{n}^{k,m}\right)^{-1}\left\langle f\left(u^{k-m+l,l},t_{k-m+l},s_{l}\right),\phi_{n}\right\rangle \phi_{n}\\
+\sum_{n=1}^{\infty}\left(\mu_{n}^{k,m}\right)^{-1}\left\langle \beta^{k-m},\phi_{n}\right\rangle \mu_{n}^{k-m,0}\phi_{n} & ,k>m,\\
\sum_{n=1}^{\infty}\omega\sum_{l=1}^{m}\mu_{n}^{l,m-k+l}\left(\mu_{n}^{k,m}\right)^{-1}\left\langle f\left(u^{l,m-k+l},t_{l},s_{m-k+l}\right),\phi_{n}\right\rangle \phi_{n}\\
+\sum_{n=1}^{\infty}\left(\mu_{n}^{k,m}\right)^{-1}\left\langle \alpha^{m-k},\phi_{n}\right\rangle \mu_{n}^{0,m-k}\phi_{n} & ,m>k.
\end{cases}\label{eq:22}
\end{equation}

From now on, we shall give an iterative scheme by knowledge of linear
approximation.

Choosing $u_{0}^{k,m}=0$, we seek $u_{q}^{k,m},q\ge1$ satisfying

\begin{equation}
u_{q}^{k,m}=\begin{cases}
\sum_{n=1}^{\infty}\omega\sum_{l=1}^{m}\mu_{n}^{k-m+l,l}\left(\mu_{n}^{k,m}\right)^{-1}\left\langle f\left(u_{q-1}^{k-m+l,l},t_{k-m+l},s_{l}\right),\phi_{n}\right\rangle \phi_{n}\\
+\sum_{n=1}^{\infty}\left(\mu_{n}^{k,m}\right)^{-1}\left\langle \beta^{k-m},\phi_{n}\right\rangle \mu_{n}^{k-m,0}\phi_{n} & ,k>m,\\
\sum_{n=1}^{\infty}\omega\sum_{l=1}^{k_{1}}\mu_{n}^{l,m-k+l}\left(\mu_{n}^{k,m}\right)^{-1}\left\langle f\left(u_{q-1}^{l,m-k+l},t_{l},s_{m-k+l}\right),\phi_{n}\right\rangle \phi_{n}\\
+\sum_{n=1}^{\infty}\left(\mu_{n}^{k,m}\right)^{-1}\left\langle \alpha^{m-k},\phi_{n}\right\rangle \mu_{n}^{0,m-k}\phi_{n} & ,m>k.
\end{cases}\label{eq:23}
\end{equation}

Here $u_{q}^{k,m}$ is called the approximate solution for the problem
(\ref{eq:1})-(\ref{eq:2}). Our results are to prove that this solution
approach to the discrete solution $u^{k,m}$ in norm $\mathcal{H}$
as $q\to\infty$ and study the stability estimate of $u_{q}^{k,m}$
in norm $\mathcal{H}$ with respect to the initial data and the right
hand side $f\left(u_{q-1}\right)$.
\begin{thm}
Let $\left\{ u_{q}^{k,m}\right\} _{q\ge1}$ be the iterative sequence
defined by (\ref{eq:23}). Then, it satisfies the a priori estimate

\[
\sup_{1\le k,m\le M}\left\Vert u_{q}^{k,m}\right\Vert ^{2}\le C_{T}\left(\sup_{1\le k,m\le M}\left\Vert u_{q-1}^{k,m}\right\Vert ^{2}+\sup_{1\le k,m\le M}\left\Vert f\left(0,t_{k},s_{m}\right)\right\Vert ^{2}+\sup_{0\le m\le M}\left\Vert \alpha^{m}\right\Vert ^{2}+\sup_{0\le k\le M}\left\Vert \beta^{k}\right\Vert ^{2}\right),
\]

where $C_{T}$ is a positive constant depending only on $T$.
\end{thm}

\begin{thm}
If the nonlinear source function $f$ of the problem $\left(P\right)$
satisfying the Lipschitz condition (\ref{eq:18}). Then, the iterative
sequence $\left\{ u_{q}^{k,m}\right\} $ defined by (\ref{eq:23})
strongly converges to the discrete solution $u^{k,m}$ (\ref{eq:22})
of $\left(P\right)$ in norm $\mathcal{H}$ in the sense of

\[
\sup_{1\le k,m\le M}\left\Vert u_{q}^{k,m}-u^{k,m}\right\Vert \le\frac{\kappa_{T}^{q}}{1-\kappa_{T}}\sup_{1\le k,m\le M}\left\Vert u_{1}^{k,m}\right\Vert ,
\]

where $\kappa_{T}<1$ is a positive constant depending only on $T$.
\end{thm}

The problem (\ref{eq:1})-(\ref{eq:2})
with the nonlinear function $f$ in a bit larger class in terms of non-Lipschitz
functions can be solved in the similar way. Specifically,
the function $f$ will be defined by the product of two 
functions $g$ and $h$. We also study the a priori estimate of solution
and obtain convergence rate between the approximate solution $u_{q}^{k,m}$
and the discrete solution $u^{k,m}$. This study continuously contributes
to the state of rarity of numerical methods for the nonlinear ultraparabolic
problems. In particular, we shall consider the following problem.

\[
\left(Q\right):\quad\begin{cases}
\frac{\partial}{\partial t}u\left(t,s\right)+\frac{\partial}{\partial s}u\left(t,s\right)+\mathcal{L}u\left(t,s\right)=g\left(u\left(t,s\right),t,s\right)h\left(u\left(t,s\right),t,s\right),\\
u\left(0,s\right)=\alpha\left(s\right),\quad u\left(t,0\right)=\beta\left(t\right),
\end{cases}
\]

for $g$ and $h$ satisfying the following conditions.

\begin{equation}
 g\left(u,t,s\right) \le K_{1},\quad\left\Vert h\left(u,t,s\right)-h\left(v,t,s\right)\right\Vert \le K_{2}\left\Vert u-v\right\Vert ,\label{eq:24}
\end{equation}

where $K_{1},K_{2}$ are positive constants independent of $u,v,t,s$.

Similar to the problem $\left(P\right)$, the approximate solution
$u_{q}^{k,m}$ of $\left(Q\right)$ is given by

\begin{equation}
u_{q}^{k,m}=\begin{cases}
\sum_{n=1}^{\infty}\omega\sum_{l=1}^{m}\mu_{n}^{k-m+l,l}\left(\mu_{n}^{k,m}\right)^{-1}\left\langle \left(gh\right)\left(u_{q-1}^{k-m+l,l},t_{k-m+l},s_{l}\right),\phi_{n}\right\rangle \phi_{n}\\
+\sum_{n=1}^{\infty}\left(\mu_{n}^{k,m}\right)^{-1}\left\langle \beta^{k-m},\phi_{n}\right\rangle \mu_{n}^{k-m,0}\phi_{n} & ,k>m,\\
\sum_{n=1}^{\infty}\omega\sum_{l=1}^{k_{1}}\mu_{n}^{l,m-k+l}\left(\mu_{n}^{k,m}\right)^{-1}\left\langle \left(gh\right)\left(u_{q-1}^{l,m-k+l},t_{l},s_{m-k+l}\right),\phi_{n}\right\rangle \phi_{n}\\
+\sum_{n=1}^{\infty}\left(\mu_{n}^{k,m}\right)^{-1}\left\langle \alpha^{m-k},\phi_{n}\right\rangle \mu_{n}^{0,m-k}\phi_{n} & ,m>k.
\end{cases}\label{eq:25}
\end{equation}

\begin{thm}
Let $\left\{ u_{q}^{k,m}\right\} _{q\ge1}$ be the iterative sequence
defined by (\ref{eq:25}). Then, it satisfies the a priori estimate

\[
\sup_{1\le k,m\le M}\left\Vert u_{q}^{k,m}\right\Vert ^{2} \le  C_{T}\left(\sup_{1\le k,m\le M}\left\Vert u_{q-1}^{k,m}\right\Vert ^{2}+\sup_{1\le k,m\le M}\left\Vert h\left(0,t_{k},s_{m}\right)\right\Vert ^{2}+\sup_{0\le m\le M}\left\Vert \alpha^{m}\right\Vert ^{2}+\sup_{0\le k\le M}\left\Vert \beta^{k}\right\Vert ^{2}\right),
\]

where $C_{T}$ is a positive constant depending only on $T$.
\end{thm}

\begin{thm}
If the nonlinear source function $f=gh$ of the problem $\left(Q\right)$
satisfying the conditions (\ref{eq:24}), then the iterative sequence $\left\{ u_{q}^{k,m}\right\} $
defined in (\ref{eq:25}) strongly converges to the discrete solution
$u^{k,m}$ (\ref{eq:22}) of $\left(Q\right)$ in norm $\mathcal{H}$
in the sense of

\[
\left\Vert u_{q}^{k,m}-u^{k,m}\right\Vert \le\kappa_{T}^{q}\sup_{1\le k,m\le M}\left\Vert u_{1}^{k,m}\right\Vert,
\]

where $\kappa_{T}<1$ is a postive constant depending only on $T$.
\end{thm}

\begin{rem}
Assume that $u_{ex}^{k,m}$ is the exact solution of the problem $\left(P\right)$
(also $\left(Q\right)$) at $\left(t_{k},s_{m}\right)$. By finite
difference scheme, we know that the error between the exact solution
$u_{ex}^{k,m}$ and discrete solution $u^{k,m}$ is of order $\mathcal{O}\left(\omega\right)$.
Therefore, by triangle inequality, the error estimate between the
exact solution and approximate solution $u_{q}^{k,m}$ is

\[
\sup_{1\le k,m\le M}\left\Vert u_{ex}^{k,m}-u_{q}^{k,m}\right\Vert \le C\left(\omega+\kappa_{T}^{q}\right),
\]

for the problem $\left(P\right)$ (also $\left(Q\right)$) where $C$
is a positive constant independent of $\omega,\kappa_{T}$ and $q$,
provided by the smoothness of the exact function.
\end{rem}

\subsection*{Proof of Theorem 4.}
\begin{proof}
Using Parseval's identity in (\ref{eq:23}) for the case $k>m$, we
have

\begin{eqnarray}
\left\Vert u_{q}^{k,m}\right\Vert ^{2} & \le & 2\omega^{2}\sum_{n=1}^{\infty}\left|\sum_{l=1}^{m}\mu_{n}^{k-m+l,l}\left(\mu_{n}^{k,m}\right)^{-1}\left\langle f\left(u_{q-1}^{k-m+l,l},t_{k-m+l},s_{l}\right),\phi_{n}\right\rangle \right|^{2}\nonumber \\
 &  & +2\sum_{n=1}^{\infty}\left|\mu_{n}^{k-m,0}\left(\mu_{n}^{k,m}\right)^{-1}\left\langle \beta^{k-m},\phi_{n}\right\rangle \right|^{2}\nonumber \\
 & \le & 2\omega^{2}m^{2}\sup_{1\le k,m\le M}\sum_{n=1}^{\infty}\left|\left\langle f\left(u_{q-1}^{k,m},t_{k},s_{m}\right),\phi_{n}\right\rangle \right|^{2}\nonumber \\
 &  & +2\sup_{0\le k\le M}\sum_{n=1}^{\infty}\left|\left\langle \beta^{k},\phi_{n}\right\rangle \right|^{2}\nonumber \\
 & \le & 2T^{2}\sup_{1\le k,m\le M}\left\Vert f\left(u_{q-1}^{k,m},t_{k},s_{m}\right)\right\Vert ^{2}+2\sup_{0\le k\le M}\left\Vert \beta^{k}\right\Vert ^{2}.\label{eq:26}
\end{eqnarray}

Similarly, we can deduce for the case $k<m$ that

\begin{equation}
\left\Vert u_{q}^{k,m}\right\Vert ^{2}\le2T^{2}\sup_{1\le k,m\le M}\left\Vert f\left(u_{q-1}^{k,m},t_{k},s_{m}\right)\right\Vert ^{2}+2\sup_{0\le m\le M}\left\Vert \alpha^{m}\right\Vert ^{2}.\label{eq:27}
\end{equation}

Moreover, from the condition (\ref{eq:18}), we have

\begin{eqnarray}
\left\Vert f\left(u_{q-1}^{k,m},t_{k},s_{m}\right)\right\Vert  & \le & \left\Vert f\left(u_{q-1}^{k,m},t_{k},s_{m}\right)-f\left(u_{0}^{k,m},t_{k},s_{m}\right)\right\Vert +\left\Vert f\left(u_{0}^{k,m},t_{k},s_{m}\right)\right\Vert \nonumber \\
 & \le & K\left\Vert u_{q-1}^{k,m}\right\Vert +\left\Vert f\left(0,t_{k},s_{m}\right)\right\Vert .\label{eq:28}
\end{eqnarray}

Combining (\ref{eq:26})-(\ref{eq:28}) and putting $C_{T}=\max\left\{ 2T^{2}\left(K^{2}+1\right);2\right\} $,
we get

\begin{eqnarray*}
\sup_{1\le k,m\le M}\left\Vert u_{q}^{k,m}\right\Vert ^{2} & \le & 2T^{2}\left(K^{2}+1\right)\left(\sup_{1\le k,m\le M}\left\Vert u_{q-1}^{k,m}\right\Vert ^{2}+\sup_{1\le k,m\le M}\left\Vert f\left(0,t_{k},s_{m}\right)\right\Vert ^{2}\right)\\
 &  & +2\left(\sup_{0\le m\le M}\left\Vert \alpha^{m}\right\Vert ^{2}+\sup_{0\le k\le M}\left\Vert \beta^{k}\right\Vert ^{2}\right)\\
 & \le & C_{T}\left(\sup_{1\le k,m\le M}\left\Vert u_{q-1}^{k,m}\right\Vert ^{2}+\sup_{1\le k,m\le M}\left\Vert f\left(0,t_{k},s_{m}\right)\right\Vert ^{2}\right.\\
 &  & \left.+\sup_{0\le m\le M}\left\Vert \alpha^{m}\right\Vert ^{2}+\sup_{0\le k\le M}\left\Vert \beta^{k}\right\Vert ^{2}\right).
\end{eqnarray*}

\end{proof}

\subsection*{Proof of Theorem 5.}
\begin{proof}
For short, we denote $f^{k,m}\left(u^{k,m}\right)=f\left(u^{k,m},t_{k},s_{m}\right)$.
Putting $w_{q}^{k,m}=u_{q+1}^{k,m}-u_{q}^{k,m}$, it follows from
(\ref{eq:23}) that for $k>m$ we have

\begin{eqnarray*}
\left\Vert w_{q}^{k,m}\right\Vert ^{2} & \le & \omega^{2}\sum_{n=1}^{\infty}\left|\sum_{l=1}^{m}\left\langle f^{k-m+l,l}\left(u_{q}^{k-m+l,l}\right)-f^{k-m+l,l}\left(u_{q-1}^{k-m+l,l}\right),\phi_{n}\right\rangle \right|^{2}\\
 & \le & \omega^{2}m^{2}\sup_{1\le k,m\le M}\left\Vert f^{k,m}\left(u_{q}^{k,m}\right)-f^{k,m}\left(u_{q-1}^{k,m}\right)\right\Vert ^{2}\\
 & \le & T^{2}K^{2}\sup_{1\le k,m\le M}\left\Vert u_{q}^{k,m}-u_{q-1}^{k,m}\right\Vert ^{2}.
\end{eqnarray*}

Thus, we get

\[
\left\Vert w_{q}^{k,m}\right\Vert \le TK\sup_{1\le k,m\le M}\left\Vert w_{q-1}^{k,m}\right\Vert .
\]

We can always choose $T>0$ small enough such that $\kappa_{T}:=TK<1$.
Then, we have

\begin{eqnarray*}
\left\Vert u_{q+r}^{k,m}-u_{q}^{k,m}\right\Vert  & \le & \left\Vert u_{q+r}^{k,m}-u_{q+r-1}^{k,m}\right\Vert +...+\left\Vert u_{q+1}^{k,m}-u_{q}^{k,m}\right\Vert \\
 & \le & \kappa_{T}^{q+r-1}\sup_{1\le k,m\le M}\left\Vert u_{1}^{k,m}-u_{0}^{k,m}\right\Vert +...+\kappa_{T}^{q}\sup_{1\le k,m\le M}\left\Vert u_{1}^{k,m}-u_{0}^{k,m}\right\Vert \\
 & \le & \kappa_{T}^{q}\left(\kappa_{T}^{r-1}+\kappa_{T}^{r-2}+...+1\right)\sup_{1\le k,m\le M}\left\Vert u_{1}^{k,m}\right\Vert \\
 & \le & \frac{\kappa_{T}^{q}\left(1-\kappa_{T}^{r}\right)}{1-\kappa_{T}}\sup_{1\le k,m\le M}\left\Vert u_{1}^{k,m}\right\Vert .
\end{eqnarray*}

Therefore, we obtain

\begin{equation}
\left\Vert u_{q+r}^{k,m}-u_{q}^{k,m}\right\Vert \le\frac{\kappa_{T}^{q}}{1-\kappa_{T}}\sup_{1\le k,m\le M}\left\Vert u_{1}^{k,m}\right\Vert ,\label{eq:29}
\end{equation}

which leads to the claim that $\left\{ u_{q}^{k,m}\right\} $ is a
Cauchy sequence in $\mathcal{H}$ and then, there exists uniquely
$u^{k,m}\in\mathcal{H}$ such that $u_{q}^{k,m}\to u^{k,m}$ as $q\to\infty$.
Because of this convergence and Lipschitz property (\ref{eq:18})
of nonlinear source term $f$, it is easy to prove that $f\left(u_{q}^{k,m}\right)\to f\left(u^{k,m}\right)$
as $q\to\infty$. Therefore, $u^{k,m}$ is the discrete solution of
the problem $\left(P\right)$.

When $r\to\infty$, it follows (\ref{eq:29}) from that

\[
\left\Vert u^{k,m}-u_{q}^{k,m}\right\Vert \le\frac{\kappa_{T}^{q}}{1-\kappa_{T}}\sup_{1\le k,m\le M}\left\Vert u_{1}^{k,m}\right\Vert .
\]

For $m>k$, we also have a similar proof. Hence, we complete the proof
of the theorem.
\end{proof}

\section{Numerical examples}

In this section, we are going to show four numerical examples in order
to validate the efficiency of our scheme. It will be observed by comparing
the results between numerical and exact solutions. We shall choose
given functions in such a way that they lead to a given exact solution.
In details, we have four examples implementing all considered cases.
The first and second examples are of the linear nonhomogeneous case
while the rest of examples are showed for nonlinear cases $f=\sin\left(u\right)$
and $f=u\sin\left(\dfrac{u}{2}\right)$ implying Lipschitz and non-Lipschitz
functions. The examples are involved with the Hilbert space $\mathcal{H}=L^{2}\left(0,\pi\right)$
and associated with homogeneous boundary conditions. On the other
hand, numerical results with many 3-D graphs shall be discussed in
the last subsection.

\subsection{Example 1}

We consider the problem

\[
\begin{cases}
u_{t}\left(x,t,s\right)+u_{s}\left(x,t,s\right)-u_{xx}\left(x,t,s\right)=f\left(x,t,s\right) & ,\left(x,t,s\right)\in\left(0,\pi\right)\times\left(0,1\right)\times\left(0,1\right),\\
u\left(0,t,s\right)=u\left(\pi,t,s\right)=0 & ,\left(t,s\right)\in\left[0,1\right]\times\left[0,1\right],\\
u\left(x,0,s\right)=\alpha\left(x,s\right) & ,\left(x,s\right)\in\left[0,\pi\right]\times\left[0,1\right],\\
u\left(x,t,0\right)=\beta\left(x,t\right) & ,\left(x,t\right)\in\left[0,\pi\right]\times\left[0,1\right],
\end{cases}
\]

where $f\left(x,t,s\right)=-2e^{-2t-s}\sin x,\alpha\left(x,s\right)=e^{-s}\sin x$
and $\beta\left(x,t\right)=e^{-2t}\sin x$.

In this example, we see that $\mathcal{D}\left(\mathcal{L}\right)=H_{0}^{1}\left(0,\pi\right)\cap H^{2}\left(0,\pi\right)$,
then we get an orthonormal eigenbasis $\phi_{n}\left(x\right)=\sqrt{\dfrac{2}{\pi}}\sin\left(\sqrt{\lambda_{n}}x\right)$
where the eigenvalues $\lambda_{n}=n^{2}$. Therefore, by time discretization
$t_{k}=k\omega,s_{m}=m\omega$ and $\omega=\dfrac{1}{M}$ the approximate
solution (\ref{eq:15}) is

\begin{equation}
u\left(x,t_{k},s_{m}\right)=\frac{-2\omega}{e^{\frac{1}{2}\left(t_{k}+s_{m}\right)}}\sum_{l=1}^{m}e^{-\frac{1}{2}\left(3t_{k-m+l}+s_{l}\right)+\frac{3}{2}\omega}\sin x+e^{-\frac{1}{2}\left(3t_{k-m}+t_{k}+s_{m}\right)}\sin x,\label{eq:31}
\end{equation}

for $k>m$ and $1\le k,m\le M$ and

\begin{equation}
u\left(x,t_{k},s_{m}\right)=\frac{-2\omega}{e^{\frac{1}{2}\left(t_{k}+s_{m}\right)}}\sum_{l=1}^{k}e^{-\frac{1}{2}\left(3t_{l}+s_{m-k+l}\right)+\frac{3}{2}\omega}\sin x+e^{-\frac{1}{2}\left(s_{m-k}+t_{k}+s_{m}\right)}\sin x,\label{eq:32}
\end{equation}

for $m>k$.

After dividing the space interval $\left[0,\pi\right]$, we have $u^{j,k,m}=u\left(x_{j},t_{k},s_{m}\right)$
the value of mesh function (\ref{eq:31})-(\ref{eq:32}) at $\left(x_{j},t_{k},s_{m}\right)$
where $x_{j}=j\tau,0\le j\le L,\tau=\dfrac{\pi}{L}$.

\subsection{Example 2}

In this example, we consider the problem

\[
\begin{cases}
u_{t}\left(x,t,s\right)+u_{s}\left(x,t,s\right)-u_{xx}\left(x,t,s\right)=f\left(x,t,s\right) & ,\left(x,t,s\right)\in\left(0,\pi\right)\times\left(0,1\right)\times\left(0,1\right),\\
u_{x}\left(0,t,s\right)=u\left(\pi,t,s\right)=0 & ,\left(t,s\right)\in\left[0,1\right]\times\left[0,1\right],\\
u\left(x,0,s\right)=\alpha\left(x,s\right) & ,\left(x,s\right)\in\left[0,\pi\right]\times\left[0,1\right],\\
u\left(x,t,0\right)=\beta\left(x,t\right) & ,\left(x,t\right)\in\left[0,\pi\right]\times\left[0,1\right],
\end{cases}
\]

where

\[
f\left(x,t,s\right)=\left[\left(\dfrac{t}{2}+2\right)^{2}+\left(\dfrac{s}{2}+2\right)^{2}\right]\cos\left(\dfrac{x}{2}\right),
\]

\[
\alpha\left(x,s\right)=\left(s^{2}+32\right)\cos\left(\dfrac{x}{2}\right),\quad\beta\left(x,t\right)=\left(t^{2}+32\right)\cos\left(\dfrac{x}{2}\right).
\]

Based on $\mathcal{D}\left(\mathcal{L}\right)=\left\{ v\in H^{1}\left(0,\pi\right)\cap H^{2}\left(0,\pi\right):v\left(\pi\right)=v_{x}\left(0\right)=0\right\} $,
the orthonormal eigenbasis and the eigenvalues are $\phi_{n}=\sqrt{\dfrac{2}{\pi}}\cos\left(\left(n-\dfrac{1}{2}\right)x\right)$
and $\lambda_{n}=\left(n-\dfrac{1}{2}\right)^{2}$, respectively.
Therefore, the approximate solution is given as follows.

\begin{eqnarray}
u\left(x_{j},t_{k},s_{m}\right) & = & \frac{\omega}{e^{\frac{1}{8}\left(t_{k}+s_{m}\right)}}\sum_{l=1}^{m}e^{\frac{1}{8}\left(t_{k-m+l}+s_{l}\right)}\left[\left(\frac{2t_{k-m+l}-\omega}{4}+2\right)^{2}+\left(\frac{2s_{l}-\omega}{4}+2\right)^{2}\right]\cos\left(\frac{x_{j}}{2}\right)\nonumber \\
 &  & +\left(t_{k-m}^{2}+32\right)e^{-\frac{1}{8}\left(t_{k}+s_{m}-t_{k-m}\right)}\cos\left(\frac{x_{j}}{2}\right),\quad k>m,\label{eq:33-1}
\end{eqnarray}

\begin{eqnarray}
u\left(x_{j},t_{k},s_{m}\right) & = & \frac{\omega}{e^{\frac{1}{8}\left(t_{k}+s_{m}\right)}}\sum_{l=1}^{k}e^{\frac{1}{8}\left(t_{l}+s_{m-k+l}\right)}\left[\left(\frac{2t_{l}-\omega}{4}+2\right)^{2}+\left(\frac{2s_{m-k+l}-\omega}{4}+2\right)^{2}\right]\cos\left(\frac{x_{j}}{2}\right)\nonumber \\
 &  & +\left(s_{m-k}^{2}+32\right)e^{-\frac{1}{8}\left(t_{k}+s_{m}-s_{m-k}\right)}\cos\left(\frac{x_{j}}{2}\right),\quad m>k,\label{eq:34-1}
\end{eqnarray}

where $t_{k}=k\omega,s_{m}=m\omega,\omega=\dfrac{1}{M},1\le k,m\le M$
and $x_{j}=j\tau,\tau=\dfrac{\pi}{L},0\le j\le L$.

\subsection{Example 3}

Now we take the following problem as an example for the nonlinear
case with $f$ the Lipschitz function.

\[
\begin{cases}
u_{t}\left(x,t,s\right)+u_{s}\left(x,t,s\right)-u_{xx}\left(x,t,s\right)+u\left(x,t,s\right)=f\left(u,x,t,s\right) & ,\left(x,t,s\right)\in\left(0,\pi\right)\times\left(0,\frac{1}{4}\right)\times\left(0,\frac{1}{4}\right),\\
u\left(0,t,s\right)=u_{x}\left(\pi,t,s\right)=0 & ,\left(t,s\right)\in\left[0,\frac{1}{4}\right]\times\left[0,\frac{1}{4}\right],\\
u\left(x,0,s\right)=\alpha\left(x,s\right) & ,\left(x,s\right)\in\left[0,\pi\right]\times\left[0,\frac{1}{4}\right],\\
u\left(x,t,0\right)=\beta\left(x,t\right) & ,\left(x,t\right)\in\left[0,\pi\right]\times\left[0,\frac{1}{4}\right],
\end{cases}
\]

where

\[
f\left(u,x,t,s\right)=\frac{1}{4}\left(\sin\left(u\right)+49u_{ex}\left(x,t,s\right)-\sin\left(u_{ex}\left(x,t,s\right)\right)\right),
\]

\[
\alpha\left(x,s\right)=\frac{1}{4}\left(1+e^{-s}\right)\sin\left(\frac{7x}{2}\right),\quad\beta\left(x,t\right)=\frac{1}{4}\left(e^{-t}+1\right)\sin\left(\frac{7x}{2}\right).
\]

With the operator $\mathcal{L}=-\dfrac{\partial}{\partial x^{2}}+I$
and $\mathcal{D}\left(\mathcal{L}\right)=\left\{ v\in H^{1}\left(0,\pi\right)\cap H^{2}\left(0,\pi\right):v\left(0\right)=v_{x}\left(\pi\right)=0\right\} $,
we get $\phi_{n}=\sqrt{\dfrac{2}{\pi}}\sin\left(\left(n+\dfrac{1}{2}\right)x\right)$
and $\lambda_{n}=\left(n+\dfrac{1}{2}\right)^{2}+1$. We shall give
the iterative scheme (\ref{eq:23}) to get the approximate solution
$u_{q}^{k,m}$ by the following steps.

\textbf{Step 1.} With $q=0$: $u_{0}^{k,m}=0,1\le k,m\le M$.

\textbf{Step 2.} Let proceed to the $\left(q-1\right)$ - time, we
get $u_{q-1}^{k,m},1\le k,m\le M$. Then we shall obtain $u_{q}^{k,m},1\le k,m\le M$
as follows.

For $k>m$:

\begin{eqnarray}
u_{q}^{k,m}\left(x\right) & = & \frac{\omega}{e^{\frac{53}{8}\left(t_{k}+s_{m}\right)}}\sum_{l=1}^{m}e^{\frac{53}{8}\left(t_{k-m+l}+s_{l}\right)}R_{q-1}^{k-m+l,l}\sin\left(\frac{7x}{2}\right)\nonumber \\
 &  & +\frac{1}{4}e^{-\frac{53}{8}\left(t_{k}+s_{m}\right)}e^{\frac{53}{8}t_{k-m}}\left(e^{-t_{k-m}}+1\right)\sin\left(\frac{7x}{2}\right),\label{eq:35-1}
\end{eqnarray}

where

\begin{equation}
R_{q-1}^{k-m+l,l}=\frac{1}{2\pi}\left(\int_{0}^{\pi}\left(\sin\left(u_{q-1}^{k-m+l,l}\left(x\right)\right)-\sin\left(u_{ex}^{k-m+l,l}\left(x\right)\right)\right)\sin\left(\frac{7x}{2}\right)dx\right)+\frac{49}{16}\left(e^{-t_{k-m+l}}+e^{-s_{l}}\right).\label{eq:33}
\end{equation}

For $m>k$:

\begin{eqnarray}
u_{q}^{k,m}\left(x\right) & = & \frac{\omega}{e^{\frac{53}{8}\left(t_{k}+s_{m}\right)}}\sum_{l=1}^{k}e^{\frac{53}{8}\left(t_{l}+s_{m-k+l}\right)}R_{q-1}^{l,m-k+l}\sin\left(\frac{7x}{2}\right)\nonumber \\
 &  & +\frac{1}{4}e^{-\frac{53}{8}\left(t_{k}+s_{m}\right)}e^{\frac{53}{8}s_{m-k}}\left(1+e^{-s_{m-k}}\right)\sin\left(\frac{7x}{2}\right),\label{eq:37}
\end{eqnarray}

where

\begin{equation}
R_{q-1}^{l,m-k+l}=\frac{1}{2\pi}\left(\int_{0}^{\pi}\left(\sin\left(u_{q-1}^{l,m-k+l}\left(x\right)\right)-\sin\left(u_{ex}^{l,m-k+l}\left(x\right)\right)\right)\sin\left(\frac{7x}{2}\right)dx\right)+\frac{49}{16}\left(e^{-t_{l}}+e^{-s_{m-k+l}}\right).\label{eq:34}
\end{equation}

Since (\ref{eq:33}) and (\ref{eq:34}) are hard to compute, we shall
approximate them by using Gauss-Legendre quadrature method (see e.g.
\cite{key-14}). Particularly, they can be determined in the following
form.

\[
\int_{0}^{\pi}H\left(x\right)\sin\left(\frac{7x}{2}\right)dx=\sum_{j=0}^{j_{0}}w_{j}H\left(x_{j}\right)\sin\left(\frac{7x_{j}}{2}\right),
\]

where $x_{j}$ are abscissae in $\left[0,\pi\right]$ and $w_{j}$
are corresponding weights, $j_{0}\in\mathbb{N}$ is a given constant.

\subsection{Example 4}

We shall consider the following example

\[
\begin{cases}
u_{t}\left(x,t,s\right)+u_{s}\left(x,t,s\right)-u_{xx}\left(x,t,s\right)+2u\left(x,t,s\right)=f\left(u,x,t,s\right) & ,\left(x,t,s\right)\in\left(0,\pi\right)\times\left(0,\frac{1}{10}\right)\times\left(0,\frac{1}{10}\right),\\
u_{x}\left(0,t,s\right)=u_{x}\left(\pi,t,s\right)=0 & ,\left(t,s\right)\in\left[0,\frac{1}{10}\right]\times\left[0,\frac{1}{10}\right],\\
u\left(x,0,s\right)=\alpha\left(x,s\right) & ,\left(x,s\right)\in\left[0,\pi\right]\times\left[0,\frac{1}{10}\right],\\
u\left(x,t,0\right)=\beta\left(x,t\right) & ,\left(x,t\right)\in\left[0,\pi\right]\times\left[0,\frac{1}{10}\right],
\end{cases}
\]

where

\[
f\left(u,x,t,s\right)=u\sin\left(\frac{u}{2}\right)-u_{ex}\sin\left(\frac{u_{ex}}{2}\right)+\left(11\sin t+\cos t+10e^{-s}+11\right)\cos\left(3x\right),
\]

\[
\alpha\left(x,s\right)=\left(1+e^{-s}\right)\cos\left(3x\right),\quad\beta\left(x,t\right)=\left(\sin t+2\right)\cos\left(3x\right),
\]

In this example, we can deduce the orthonormal eigenbasis $\phi_{n}=\sqrt{\dfrac{2}{\pi}}\cos\left(nx\right)$
and the eigenvalues $\lambda_{n}=n^{2}+2$. Here the nonlinear function
$f\left(u\right)=u\sin\left(\dfrac{u}{2}\right)$ implies $g\left(u\right)=\sin\left(\dfrac{u}{2}\right)$
and $h\left(u\right)=u$ satisfying the theoretical assumptions. We
shall construct an approximate solution by steps like in Example 3.

For $k>m$:

\begin{eqnarray}
u_{q}^{k,m}\left(x\right) & = & \frac{\omega}{e^{\frac{11}{2}\left(t_{k}+s_{m}\right)}}\sum_{l=1}^{m}e^{\frac{11}{2}\left(t_{k-m+l}+s_{l}\right)}R_{q-1}^{k-m+l,l}\cos\left(3x\right)\nonumber \\
 &  & +e^{-\frac{11}{2}\left(t_{k}+s_{m}\right)}e^{\frac{11}{2}t_{k-m}}\left(\sin\left(t_{k-m}\right)+2\right)\cos\left(3x\right),\label{eq:39}
\end{eqnarray}

where

\begin{eqnarray*}
R_{q-1}^{k-m+l,l} & = & \frac{2}{\pi}\int_{0}^{\pi}u_{q-1}^{k-m+l,l}\left(x\right)\sin\left(\frac{u_{q-1}^{k-m+l,l}\left(x\right)}{2}\right)\cos\left(3x\right)dx\\
 &  & -\frac{2}{\pi}\int_{0}^{\pi}u_{ex}^{k-m+l,l}\left(x\right)\sin\left(\frac{u_{ex}^{k-m+l,l}\left(x\right)}{2}\right)\cos\left(3x\right)dx\\
 &  & +11\sin\left(t_{k-m+l}\right)+\cos\left(t_{k-m+l}\right)+10e^{-s_{l}}+11.
\end{eqnarray*}

For $m>k$:

\begin{eqnarray}
u_{q}^{k,m}\left(x\right) & = & \frac{\omega}{e^{\frac{11}{2}\left(t_{k}+s_{m}\right)}}\sum_{l=1}^{k}e^{\frac{11}{2}\left(t_{l}+s_{m-k+l}\right)}R_{q-1}^{l,m-k+l}\cos\left(3x\right)\nonumber \\
 &  & +e^{-\frac{11}{2}\left(t_{k}+s_{m}\right)}e^{\frac{11}{2}s_{m-k}}\left(1+e^{-s_{m-k}}\right)\cos\left(3x\right),\label{eq:40}
\end{eqnarray}

where

\begin{eqnarray*}
R_{q-1}^{l,m-k+l} & = & \frac{2}{\pi}\int_{0}^{\pi}u_{q-1}^{l,m-k+l}\left(x\right)\sin\left(\frac{u_{q-1}^{l,m-k+l}\left(x\right)}{2}\right)\cos\left(3x\right)dx\\
 &  & -\frac{2}{\pi}\int_{0}^{\pi}u_{ex}^{l,m-k+l}\left(x\right)\sin\left(\frac{u_{ex}^{l,m-k+l}\left(x\right)}{2}\right)\cos\left(3x\right)dx\\
 &  & +11\sin\left(t_{l}\right)+\cos\left(t_{l}\right)+10e^{-s_{m-k+l}}+11.
\end{eqnarray*}

\subsection{Discussion of results}

\begin{figure}
\subfloat[]{\includegraphics[scale=0.4]{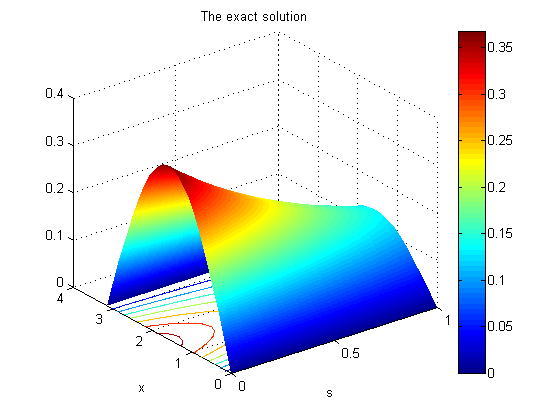}

}\subfloat[]{\includegraphics[scale=0.4]{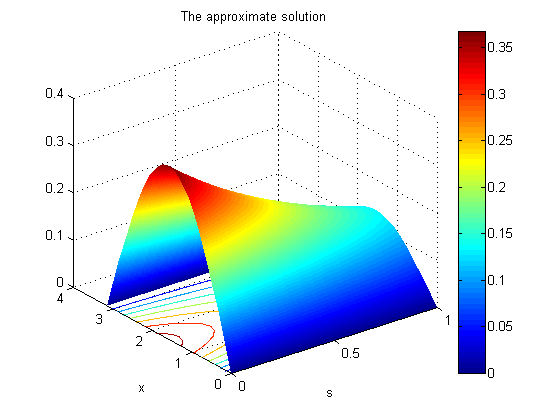}

}

\protect\caption{\small{The exact solution $u_{ex}\left(x,t,s\right)=e^{-2t-s}\sin x$ shown
in (a) in comparison with the approximate solution (\ref{eq:31})-(\ref{eq:32})
shown in (b) at $t=\dfrac{1}{2}$ for Example 1.}}
\end{figure}

\begin{figure}
\subfloat[]{\includegraphics[scale=0.4]{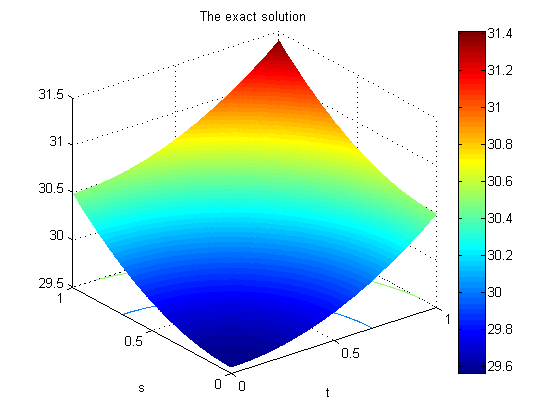}

}
\subfloat[]{\includegraphics[scale=0.4]{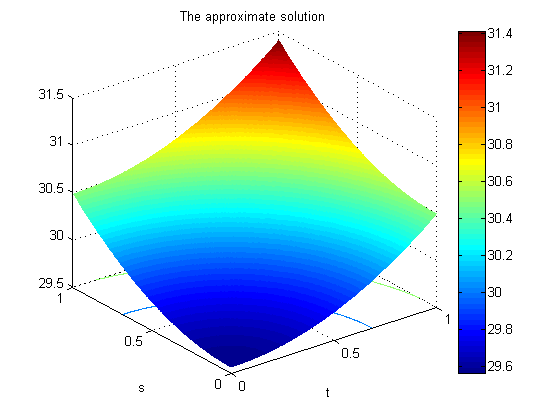}

}

\protect\caption{\small{The exact solution $u_{ex}\left(x,t,s\right)=\left(t^{2}+s^{2}+32\right)\cos\left(\dfrac{x}{2}\right)$
shown in (a) in comparison with the approximate solution (\ref{eq:33-1})-(\ref{eq:34-1})
shown in (b) at $x=\dfrac{\pi}{4}$ for Example 2.}}
\end{figure}

\begin{figure}
\subfloat[]{\includegraphics[scale=0.4]{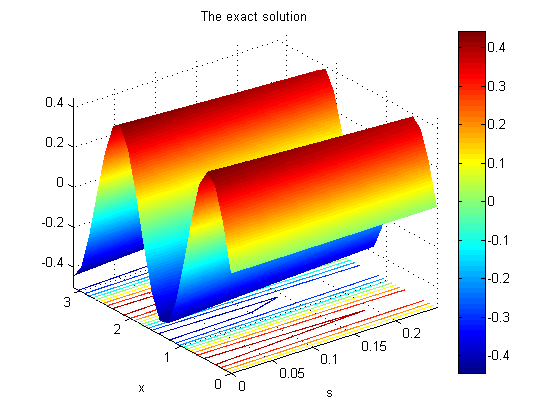}

}\subfloat[]{\includegraphics[scale=0.4]{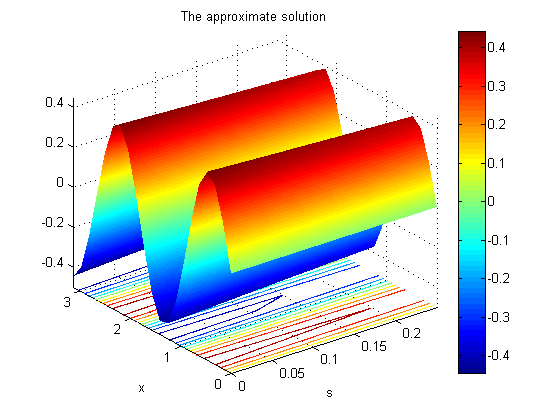}

}

\protect\caption{\small{The exact solution $u_{ex}\left(x,t,s\right)=\dfrac{1}{4}\left(e^{-t}+e^{-s}\right)\sin\left(\dfrac{7x}{2}\right)$
shown in (a) in comparison with the approximate solution (\ref{eq:35-1})-(\ref{eq:37})
shown in (b) at $t=\dfrac{1}{4}$ for Example 3.}}
\end{figure}

\begin{figure}
\subfloat[]{\includegraphics[scale=0.4]{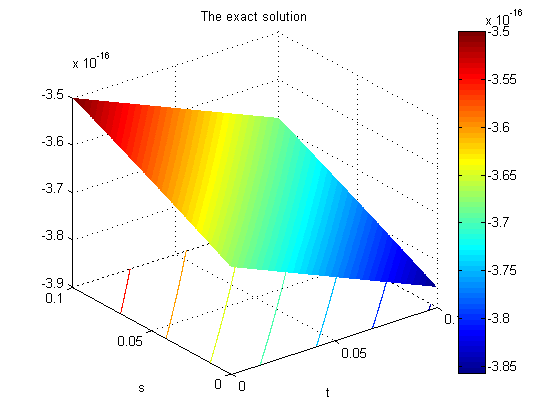}

}\subfloat[]{\includegraphics[scale=0.4]{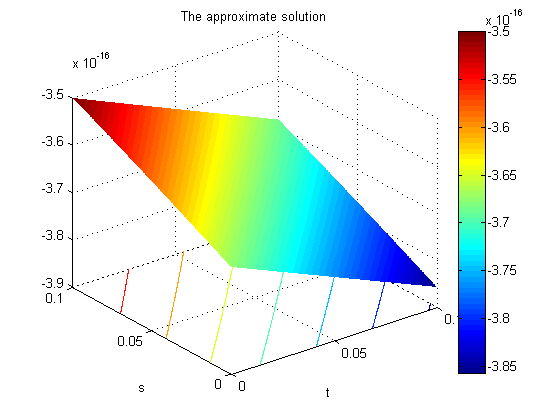}

}

\protect\caption{\small{The exact solution $u_{ex}\left(x,t,s\right)=\left(\sin t+1+e^{-s}\right)\cos\left(3x\right)$
shown in (a) in comparison with the approximate solution (\ref{eq:39})-(\ref{eq:40})
shown in (b) at $x=\dfrac{\pi}{2}$ for Example 4.}}
\end{figure}

Denoting $E=u_{ex}-u$, we compute the discrete $l_{2}$-norm and
$l_{\infty}$-norm of $E$ by

\begin{equation}
\left\Vert E\right\Vert _{l_{2}}=\sqrt{\frac{1}{\mathcal{\left|G\right|}}\sum_{\chi_{G}\in\mathcal{G}}\left|E\left(\chi_{G}\right)\right|^{2}},\quad\left\Vert E\right\Vert _{l_{\infty}}=\max_{\chi_{G}\in\mathcal{G}}\left|E\left(\chi_{G}\right)\right|,\label{eq:35}
\end{equation}

where $\mathcal{G}=\left\{ \chi_{G}\right\} $ is a set of $\left(L+1\right)M^{2}$
points on uniform grid $\left[0,\pi\right]\times\left(0,T\right]\times\left(0,T\right]$
and $\left|\mathcal{G}\right|$ cardinality of $\mathcal{G}$.

In our computations, we always fix $j_{0}=5$ and $L=20$. The comparison
between the exact solutions and the approximate solutions for the
examples respectively are shown in Figure 1-Figure 4 in graphical
representations. As in these figures, we can see that the exact solution
and the approximate solution are close together. Furthermore, convergence
is observed from the computed errors in Table 1-Table 4 for Example
1-Example 4, respectively, which is reasonable for our theoretical
results.

\begin{table}[h!]
\caption{Numerical results (\ref{eq:35}) for Example 1 with $L=20$.}
      \begin{tabular}{ccccc}
        \hline
           & $ M=50 $ & $ M=100 $ & $ M=200 $ & $ M=400 $\\ \hline
        $ \left\Vert E\right\Vert _{l_{2}} $ & 1.51608045E-03 & 7.55346287E-04 & 3.85041789E-04 & 1.88342949E-04\\
        $ \left\Vert E\right\Vert _{l_{\infty}} $ & 3.84188903E-03 & 1.92287169E-03 & 9.61845810E-04 & 4.81024266E-04\\ \hline
      \end{tabular}
\end{table}

\begin{table}
\caption{Numerical results (\ref{eq:35}) for Example 2 with $ L=20 $.}
      \begin{tabular}{ccccc}
        \hline
           & $ M=50 $ & $ M=100 $ & $ M=200 $ & $ M=400 $\\ \hline
        $ \left\Vert E\right\Vert _{l_{2}} $ & 6.53270883E-03 & 3.26622222E-03 & 1.63312276E-03 & 8.16570001E-04\\
        $ \left\Vert E\right\Vert _{l_{\infty}} $ & 2.26666504E-02 & 1.13406619E-02 & 5.67215954E-03 & 2.83653620E-03\\ \hline
      \end{tabular}
\end{table}

\begin{table}[h!]
\caption{Numerical results (\ref{eq:35}) for Example 3 with $ j_{0}=5,L=20 $.}
      \begin{tabular}{ccccc}
        \hline
           & $ q=2,M=50 $ & $ q=3,M=100 $ & $ q=4,M=200 $ & $ q=5,M=400 $\\ \hline
        $ \left\Vert E\right\Vert _{l_{2}} $ & 5.82730398E-03 & 2.90030629E-03 & 1.44171369E-03 & 7.18708022E-04\\
       $ \left\Vert E\right\Vert _{l_{\infty}} $ & 1.16425665E-02 & 5.85110603E-03 & 2.91848574E-03 & 1.45728672E-03\\ \hline
      \end{tabular}
\end{table}

\begin{table}[h!]
\caption{Numerical results (\ref{eq:35}) for Example 4 with $ j_{0}=5,L=20 $.}
      \begin{tabular}{ccccc}
        \hline
           & $ q=2,M=50 $ & $ q=3,M=100 $ & $ q=4,M=200 $ & $ q=5,M=400 $\\ \hline
        $ \left\Vert E\right\Vert _{l_{2}} $ & 5.45295363E-03 & 2.69804976E-03 & 1.34196386E-03 & 6.69222399E-04\\
        $ \left\Vert E\right\Vert _{l_{\infty}} $ & 1.48732036E-02 & 7.42289652E-03 & 3.70802189E-03 & 1.85315435E-03\\ \hline
      \end{tabular}
\end{table}

\section{Conclusion}

In this paper, we have studied the numerical method for solving a
class of nonlinear ultraparabolic equations in abstract Hilbert spaces,
namely problem (\ref{eq:1})-(\ref{eq:2}). Our numerical approach
is based on finite difference method and representation by Fourier
series. The method not only serves the ultraparabolic problems in
multi-space dimension but also deals with a wide class of nonlinear
ultraparabolic problems that many recent studies do not cover. Moreover,
it is useful in numerical simulations when one wants to construct
a stable, reliable and fast convergent approximation. Some stability
results and error estimates are obtained. Lots of numerical examples
are showed to see the efficiency of the method.

In fact, it should be stated that Fourier series expression of solution
may lead a limitation of the method for applications in a complicated
domain where a solution cannot be expressed by a certain series. On
the other hand, numerical method for a class of nonlinear equations
in a large time scale with a better convergence rate should be developed.
All of these issues will be surveyed in a further research.


\begin{backmatter}

\section*{Competing interests}
  The authors declare that they have no competing interests.

\section*{Author's contributions}
    All authors, VAK, LTL, NTYN and NHT contributed to each part of this work equally and read and approved the final version of the manuscript.

\section*{Acknowledgements}
  The authors wish to express their sincere
thanks to the referees for many constructive comments leading to the
improved version of this paper.
\end{backmatter}
\end{document}